\documentclass[oneside,reqno,12pt]{amsart}
\usepackage{amsfonts,amssymb,amsmath,amsthm}
\usepackage{setspace}

\usepackage[hyphens]{url}
\usepackage{hyperref}
\usepackage[hyphenbreaks]{breakurl}

\usepackage{centernot}
\usepackage{mathtools}
\usepackage{ stmaryrd }
\usepackage{graphicx}
\usepackage{euscript,enumerate}
\usepackage{xcolor,wrapfig}
\usepackage{pxfonts}
\usepackage{verbatim}
\usepackage{float}
\usepackage{nccmath}

\theoremstyle{plain} %--default
\newtheorem{theorem}             {Theorem} % [section]
\newtheorem{corollary}  [theorem]{Corollary}

\newtheorem{lemma}[theorem]{Lemma}
\newtheorem{conjecture}[theorem]{Conjecture}

\theoremstyle{definition}
\newtheorem{definition}{Definition}

\newtheorem{problem}{Problem}

\theoremstyle{remark}

\DeclareMathOperator{\lcm}{lcm}
\def\modd#1 #2{#1\ \mbox{\rm (mod}\ #2\mbox{\rm )}}

\begin{document}
\begin{singlespace}
\author{Chris Bispels}
\address{Department of Mathematics and Statistics, University of Maryland - Baltimore
County, Baltimore, MD 21250, USA}
\email{cbispel1@umbc.edu}

\author{Muhammet Boran}
\address{Department of Mathematics, Yildiz Technical University, 34220 Esenler, Istanbul, TURKEY}
\email{muhammet.boran@std.yildiz.edu.tr}

\author{Steven J.\ Miller}
\address{Department of Mathematics and Statistics, Williams College, Williamstown, MA 01267, USA}
\email{sjm1@williams.edu}

\author{Eliel Sosis}
\address{Department of Mathematics, University of Michigan, Ann Arbor, MI 48104, USA}
\email{esosis@umich.edu}

\author{Daniel Tsai}
\address{Department of Mathematics, National Taiwan University, No. 1, Sec. 4, Roosevelt Rd., Taipei 10617, Taiwan (R.O.C.)}
\email{tsaidaniel@ntu.edu.tw}

\title{$v$-Palindromes: An Analogy to the Palindromes}

\begin{abstract}
Around the year 2007, one of the authors, Tsai, accidentally discovered a property of the number $198$ he saw on the license plate of a car. Namely, if we take $198$ and its reversal $891$, which have prime factorizations $198 = 2\cdot 3^2\cdot 11$ and $891 = 3^4\cdot 11$ respectively, and sum the numbers appearing in each factorization getting $2+3+2+11 = 18$ and $3+4+11 = 18$, both sums are $18$. Such numbers were later named $v$-palindromes because they can be viewed as an analogy to the usual palindromes. In this article, we introduce the concept of a $v$-palindrome in base $b$ and prove their existence for infinitely many bases. We also exhibit infinite families of $v$-palindromes in bases $p+1$ and $p^2+1$, for each odd prime $p$. Finally, we collect some conjectures and problems involving $v$-palindromes.
\end{abstract}

\subjclass[2010]{Primary 11A63; Secondary 11A25, 11A51}
\thanks{This work was supported in part by the 2022 Polymath Jr REU program.}

\maketitle

\section{Introduction}

If I (D.\ Tsai) recall correctly, it was in the year 2007 when I was 15 years old. My mother and younger brother were in a video rental shop near our home in Taipei and my father and I were waiting outside the shop, standing beside our parked car. I was a bit bored and glanced at the license plate of our car, which was 0198-QB. For no clear reason, I took the number $198$ and did the following. I factorized $198=2\cdot 3^2\cdot 11$, reversed the digits of $198$, and factorized $891=3^4\cdot 11$. Then, I summed the numbers appearing in each factorization: $2+3+2+11=18$ and $3+4+11=18$, respectively. Surprisingly to me, they are equal! We also illustrate this pictorially in Figure \eqref{fig:1}.
\begin{figure}[h]
\begin{align}
198\quad&=\quad 2\cdot 3^2\cdot 11\quad\xmapsto{\phantom{aaaa}}\quad 2+(3+2)+11 \nonumber \\
\uparrow\phantom{!}\quad&\phantom{=}\quad \phantom{2\cdot 3^2\cdot 11}\quad\phantom{\xmapsto{\phantom{aaaa}}}\quad \phantom{2+(3+2}\parallel\nonumber\\
\text{digit reversal}&\phantom{=}\quad \phantom{2\cdot 3^2\cdot 11}\quad\phantom{\xmapsto{\phantom{aaaa}}}\quad \phantom{2+(3+2}18 \nonumber \\
\downarrow\phantom{!}\quad&\phantom{=}\quad \phantom{2\cdot 3^2\cdot 11}\quad\phantom{\xmapsto{\phantom{aaaa}}}\quad \phantom{2+(3+2}\parallel\nonumber\\
891\quad&=\quad \phantom{2\cdot !!} 3^4\cdot 11\quad\xmapsto{\phantom{aaaa}}\quad (3+4)+11 \nonumber
\end{align}
\caption{\label{fig:1} $198$ is a $v$-palindrome in base $10$}
\end{figure}
Afterwards I spent some time trying to prove there are infinitely many such numbers (we define them rigorously in Subsection \ref{sub:defn}), but could not show it.

In October 2018, I published a one-and-a-half page note \cite{tsai18} in the S\={u}gaku Seminar magazine, which is sort of like the American Mathematical Monthly of Japan, defining $v$-palindromes and showing their infinitude. However, I recall knowing how to show their infinitude as early as the summer of 2015.

In March, 2021, I published the paper \cite{tsai}, first calling such numbers $v$-\emph{palindromes}. I proved a general theorem \cite[Theorem 1]{tsai} describing a periodic phenomenon pertaining to $v$-palindromes. Then, in August of 2022, I published \cite{tsai22} (formerly \cite{tsai21pre} on arxiv), in which more in-depth investigations were done. I also wrote the manuscripts \cite{tsai3,thispaper,tsai5} (\cite{thispaper} is a former version of this manuscript) which are, at time of writing, preprints.

In Subsection \ref{sub:defn} we define $v$-palindromes in a general base, briefly discuss prime $v$-palindromes, and explain why the $v$-palindromes can be considered an analogue to the usual palindromes. In Subsection \ref{sec:outline}, we give an outline of the rest of the paper. In Subsection \ref{sec:other}, we review other related work.

\subsection{\texorpdfstring{$v$-palindromes}{v-palindromes}}
\label{sub:defn}

Recall the base $b$ representation of a natural number where $b\geq2$ is the base. For every natural number $n$, there exist unique integers $L\geq1$ and $0\leq a_0,a_1,\ldots,a_{L-1}<b$ with $a_{L-1}\neq0$ such that
\begin{equation}\label{baseb}
  n=a_0+a_1b+\cdots+a_{L-1}b^{L-1}.
\end{equation}
We also denote this as $n=(a_{L-1},\cdots,a_1,a_0)_b$. Thus $L$ is the number of base $b$ digits of $n$. We define the \textit{digit reversal} in base $b$ of $n$ to be
\begin{equation}
r_b(n)=a_{L-1}+a_{L-2}b+\cdots+a_{0}b^{L-1}.
\end{equation}
For instance, $r_{10}(18)=81$, $r_{10}(2)=r_{10}(200)=2$, and $r_2(2) = r_2((1,0)_2) = 1$. Next we define a function $v(n)$ to denote ``summing the numbers appearing in the factorization''.
\begin{definition}
Suppose that the prime factorization of the natural number $n$ is
\begin{equation}\label{eq:fact}
  n=p^{\varepsilon_1}_1\cdots p^{\varepsilon_s}_s q_1\cdots q_t,
\end{equation}
where $s,t\geq0$ and $\varepsilon_1,\ldots,\varepsilon_s\geq2$ are integers and $p_1,\ldots,p_s,q_1,\ldots,q_t$ are distinct primes. Then we set
\begin{equation}
  v(n)=\sum^s_{i=1}(p_i+\varepsilon_i)+\sum^t_{j=1}q_j.
\end{equation}
\end{definition}
Notice that $v(n)$ is an additive function, i.e., $v(mn)=v(m)+v(n)$ whenever $m$ and $n$ are relatively prime natural numbers. The values of $v(n)$ have been created as sequence A338038 in the On-Line Encyclopedia of Integer Sequences \cite{oeis}. We can now make the following definition.
\begin{definition}[$v$-palindrome]
Let $b\geq2$ be an integer. A natural number $n$ is a $v$-\textit{palindrome} in base $b$ if
\begin{itemize}
\item[{\rm (i)}] $b\nmid n$,
\item[{\rm (ii)}] $n\neq r_b(n)$, and
\item[{\rm (iii)}] $v(n)=v(r_b(n))$.
\end{itemize}
The set of $v$-palindromes in base $b$ is denoted by $\mathbb{V}_b$.
\end{definition}
Condition (i) is included merely for the aesthetic look of $n$ and $r_b(n)$ having the same number of digits. Condition (ii) is included since if $n=r_b(n)$, then condition (iii) holds trivially, and so nothing is surprising. The sequence of $v$-palindromes in base ten has been created as sequence A338039 in \cite{oeis}. A generalization of Figure \eqref{fig:1}, with the factorizing step omitted, would be as follows:
\begin{figure}[h]
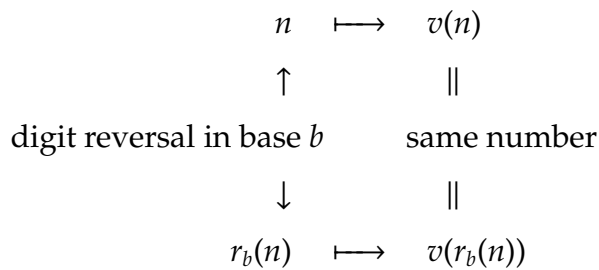

\begin{align}
n\quad&\xmapsto{\phantom{aaaa}}\quad v(n) \nonumber \\
\uparrow\quad&\phantom{\xmapsto{\phantom{aaaa}}}\quad\phantom{aa}\parallel\nonumber\\
\text{digit reversal in base $b$}&\phantom{\xmapsto{\phantom{aaaa}}}\phantom{aa}\text{same number}\nonumber\\
\downarrow\quad&\phantom{\xmapsto{\phantom{aaaa}}}\quad\phantom{aa}\parallel\nonumber\\
r_b(n)\quad&\xmapsto{\phantom{aaaa}}\quad v(r_b(n)) \nonumber
\end{align}
\caption{\label{fig:2} $v$-palindromes in base $b$}
\end{figure}

Only base ten is dealt with in \cite{tsai18,tsai,tsai22,tsai3,tsai5}, but most of the results and proofs therein generalize straightforwardly to a general base.

It is conjectured in \cite[(a)]{tsai18} that there are no prime $v$-palindromes in base ten. Recently, Boran et al.\ \cite{polymath} characterized prime $v$-palindromes in base ten and showed that they are precisely the primes of the form $5\cdot 10^m-1$ such that $5\cdot 10^m-3$ is also prime. Thus, prime $v$-palindromes in base ten, if any exist, must be twin primes.

When we consider an arbitrary base, however, we have been able to find several prime $v$-palindromes. For example, $109$ is a prime $v$-palindrome in base $16$. We have that $v(109)=109$ since $109$ is prime, and $r_{16}(109)=r_{16}((6,13)_{16})=(13,6)_{16}=214$. Then $v(r_{16}(109))=v(214)=107+2=109$, so $109$ is a $v$-palindrome in base $16$. Notice that $109$ happens to be a twin prime, but we can also find examples of prime $v$-palindromes that are not twin primes. For example, $467$ is a prime $v$-palindrome in base $276$, yet $467$ is not a twin prime. It remains an open problem to classify which bases are similar to base $10$ where any prime $v$-palindrome must necessarily be a twin prime.

We now explain how $v$-palindromes can be viewed as an analogy to usual palindromes. Recall the following definition of the usual palindrome.
\begin{definition}[palindrome]
Let $b\geq2$ be an integer. A natural number $n$ is a \emph{palindrome} in base $b$ if $n=r_b(n)$.
\end{definition}
The definition of $v$-palindromes can be obtained from that of the usual palindromes by applying $v$ to the equality $n=r_b(n)$ and then including the conditions (i) and (ii) for the reasons explained earlier. The mere application of $v$ to the equality $n=r_b(n)$ causes palindromes and $v$-palindromes to behave very differently. Although the function $v$ is a specific as function defined above, there is nothing special about it. It is equally conceivable to use any other function $f\colon\mathbb{N}\to\mathbb{C}$ instead in condition (iii), calling the defined kind of numbers $f$-palindromes.

\subsection{Outline and main results of this paper}\label{sec:outline}
In Section \ref{sec:vpal10}, we recall infinitely many more examples of $v$-palindromes in base ten from the previous works \cite{tsai18,tsai,tsai3}. In Section \ref{sec:past}, we state analogous results in a general base $b$ of results in base ten from \cite{tsai,tsai22}, as well as prove that if there exists a $v$-palindrome in a base $b$, then there exists infinitely many (Theorem \ref{thm:inf}). In Section \ref{sec:vpalb}, we exhibit $v$-palindromes in bases $p+1$ and $p^2+1$, for each odd prime $p$. In Section \ref{sec:main}, we prove that $v$-palindromes exist in infinitely many bases $b$ (Corollary \ref{cor:infb}). Finally in Section \ref{sec:further}, we discuss further problems pertaining to $v$-palindromes in a general base.

\subsection{Other related work}\label{sec:other}
In this section, we give references to other works related to $v$-palindromes in various ways. For results on the usual palindromes, we refer the reader to Goins \cite{goins}, Hern\'andez Hern\'andez and Luca \cite{santos}, Pongsriiam \cite{pong2}, Pongsriiam and Subwattanachai \cite{pong1}, and the references therein, though there is much more literature available on palindromes.

Digit reversal has been studied in the past. In Hardy \cite{hardy}, it is mentioned that $4\cdot 2178 = 8712$ and $9\cdot 1089 = 9801$. Therefore the numbers $2178$ and $1089$ have the property that their digit reversal is a multiple of (at least twice) themselves. Following Sutcliffe \cite{sutcliffe}, in general we are solving $k n = r_b(n)$ for integers $k\geq2$, $n\geq1$, and $b\geq2$. If $n$ has one base $b$ digit, then $n = r_b(n)$ and there is no solution. Hence $n$ has at least two base $b$ digits. Then by the generalization of \cite[Lemma 2.3]{polymath} to a general base, $n^2$ has more base $b$ digits than $n$. Now, as $r_b(n)$ has no more base $b$ digits than $n$, we have that $n^2$ has more base $b$ digits than $r_b(n)$. Consequently, $kn = r_b(n)$ implies that $kn=r_b(n) < n^2$, and so $k<n$. In \cite{sutcliffe}, the case of $n$ having less than four base $b$ digits is addressed completely, whereas the case of $n$ having four base $b$ digits is partially addressed. Klosinski and Smolarski \cite{klosinski} also considered this problem, and mentioned that $4\cdot 219\cdots 978 = 879\cdots 912$ for any number of $9$'s in between, generalizing the aforementioned $4\cdot 2178 = 8712$ nicely.

Other functions similar to $v(n)$ have been studied by many authors. Alladi and Erd\"{o}s \cite{alladi} and Lal \cite{lal} studied the function
\begin{equation}
  A(n) = \sum^s_{i=1}p_i \varepsilon_i+\sum^t_{j=1}q_j,
\end{equation}
following the same notation as Equation  \eqref{eq:fact}. In \cite{alladi}, analytical and other aspects of $A(n)$ are studied. In \cite{lal}, iterates of $A(n)$ are investigated. Mullin \cite{mullin} and Gordon and Robertson \cite{gordon} studied the function
\begin{equation}
  \psi(n) = \prod^s_{i=1}p_i \varepsilon_i\cdot \prod^t_{j=1}q_j.
\end{equation}
In \cite{mullin}, research problems on $\psi(n)$ were posed, and \cite{gordon} proved two theorems on $\psi(n)$.

Similar equalities to the condition $v(n) = v(r_b(n))$ from the definition of $v$-palindromes, with another function in place of $v$, have been studied by several authors. Following Spiegelhofer \cite{spieg}, we define \emph{Stern's diatomic sequence} $s(n)$ by $s(0) = 0$, $s(1) = 1$, and $s(2n) = s(n)$ and $s(2n+1) = s(n) + s(n+1)$ for all $n\geq1$. In Dijkstra \cite{dijkstra2}, a problem is given asking to show $s(n) = s(r_2(n))$ for all $n\geq1$. Dijkstra also proved this equality in \cite[pp.\ 230--232]{dijkstra}. Following \cite{spieg} again, we define the function $b(n)$ introduced by Northshield as $b(0) = 0$, $b(1) = 1$, and
\begin{align}
  b(3n) & = b(n),\\
  b(3n+1) & = \sqrt{2}\cdot  b(n) + b(n+1),\\
  b(3n+2) & = b(n) + \sqrt{2} \cdot b(n+1),
\end{align}
for all $n\geq0$. Then it is proved that $b(n) = b(r_3(n))$ for all $n\geq1$ \cite[Theorem 1]{spieg}. \cite[Theorem 2]{spieg} gives a slightly intricate sufficient condition for a complex-valued function $f(n)$ and integer $b\geq2$ to satisfy $f(n) = f(r_b(n))$ for all $n\geq1$, to which \cite[Theorem 1]{spieg} is a corollary. Following Spiegelhofer \cite{spieg2}, we define an analogue of Stern's diatomic sequence, the \emph{Stern polynomials} $s_n(x,y)$, by $s_1(x,y) = 1$ and $s_{2n}(x,y) = s_n(x,y)$ and $s_{2n+1}(x,y) = x s_n(x,y) + ys_{n+1}(x,y)$ for $n\geq1$. It was proved that $s_n(x,y) = s_{r_2(n)}(x,y)$ for all $n\geq1$ \cite[Theorem 1]{spieg2}. We introduce a final equality of this type from Morgenbesser and Spiegelhofer \cite{morgen}. Let $\sigma_b$ be the sum-of-digits function in base $b\geq2$. For $\alpha\in\mathbb{R}$ and integers $n\geq1$, define
\begin{equation}
  \gamma(\alpha,n) = \lim_{x\to\infty}\frac{1}{x}\sum_{k<x} e^{2\pi i \alpha (\sigma_b(k+n)-\sigma_b(k))}.
\end{equation}
Then it is proved that $\gamma(\alpha,n) = \gamma(\alpha,r_b(n))$ for all $n\geq1$ \cite[Theorem 1]{morgen}.

\section{More \texorpdfstring{$v$}{v}-palindromes in base ten}\label{sec:vpal10}

We illustrated in Figure \eqref{fig:1} that $198$ is a $v$-palindrome in base ten. That there are infinitely many $v$-palindromes in base ten is shown in \cite{tsai18} by specifically showing that all numbers
\begin{equation}\label{eq:1998}
18, 198, 1998,\ldots,
\end{equation}
with any number of nines in the middle are $v$-palindromes in base ten. Also, \cite{tsai18} mentions that all numbers of the form
\begin{equation}\label{eq:ten}
18,1818,181818,\ldots,
\end{equation}
with any number of $18$s concatenated, are $v$-palindromes in base ten. In fact, the main theorem of \cite{tsai} is inspired by Equation \eqref{eq:ten}. Both Equations \eqref{eq:1998} and \eqref{eq:ten} seem to be derived from the number $18$. They are subsets of the following more general family.
\begin{theorem}[{\cite[Theorem 3]{tsai3}}]
  If $\rho$ is a palindrome in base ten consisting entirely of the digits $0$ and $1$, then $18\rho$ is a $v$-palindrome in base ten.
\end{theorem}
This theorem relates the usual palindromes with the $v$-palindromes. If we take $\rho$ to be a repunit, then we get Equation \eqref{eq:1998}. If we take $\rho$ to have alternating digits of $0$ and $1$, then we get Equation \eqref{eq:ten}. If we take $\rho$ to have only the first and last digits being $1$ and at least one $0$ in between, then we deduce the family of $v$-palindromes in base ten
\begin{equation}
1818,18018,180018,\ldots,
\end{equation}
with any number of $0$'s in between two $18$'s.

Thus the infinitude of $v$-palindromes in base ten is well-established.

\section{Past results for a general base}\label{sec:past}
In this section we state some past results from \cite{tsai,tsai22}, which were just for base ten, for a general base. The base ten proofs generalize straightforwardly to a general base. Finally we show that if there exists a $v$-palindrome in a base $b$, then there exists infinitely many (Theorem \ref{thm:inf}).

\subsection{A periodic phenomenon}\label{sec:aperiodic}

We state the main theorem of \cite{tsai}, which describes a periodic phenomenon involving $v$-palindromes and repeated concatenations in base ten, for a general base. The proof in \cite{tsai} is only for base ten, but is easily adapted for a general base. Before that, we provide notation for repeated concatenations.
\begin{definition}\label{def:rc}
Suppose that $n=(a_{L-1},\ldots,a_1,a_0)_b$ is a base $b$ representation and $k\geq1$ is an integer, then we denote the repeated concatenation of the base $b$ digits of $n$ consisting of $k$ copies of $n$ by $n(k)_b$. That is,
\begin{align}
  n(k)_b & =(\underbrace{a_{L-1},\ldots, a_1,a_0,a_{L-1},\ldots, a_1,a_0,\ldots\ldots, a_{L-1},\ldots, a_1,a_0}_\text{$k$ copies of $a_{L-1},\ldots, a_1,a_0$})_b \nonumber \\
  & =n(1+b^L+\cdots + b^{(k-1)L})=n\cdot\frac{1-b^{Lk}}{1-b^L}. \label{wd}
\end{align}
\end{definition}
For instance, $18(3)_{10}=181818$ and $201(4)_{10}=201201201201$. Now we can state the main theorem of \cite{tsai} for a general base as follows.
\begin{theorem}[{\cite[Theorem 1]{tsai} for a general base}]\label{tsai1T}
Let $b\geq2$ be an integer. For every natural number $n$ with $b\nmid n$ and $n\neq r_b(n)$, there exists an integer $\omega\ge1$ such that for all integers $k\ge1$,
\begin{equation}
n(k)_b\in\mathbb{V}_b\quad\text{if and only if}\quad n(k+\omega)_b\in\mathbb{V}_b.
\end{equation}
\end{theorem}
Based on this theorem, we can make the following definitions.
\begin{definition}
The smallest possible $\omega$ in the above theorem is denoted by $\omega_0(n)_b$. If the base $b$ digits of $n$ can be repeatedly concatenated to form a $v$-palindrome in base $b$, i.e., if there exists an integer $k\geq1$ such that $n(k)_b\in \mathbb{V}_b$, then the smallest $k$ is denoted by $c(k)_b$; otherwise we set $c(n)_b=\infty$.
\end{definition}
The sequence of numbers $n$ such that $c(n)_{10}<\infty$ has been created as sequence A338371 in the On-Line Encyclopedia of Integer Sequences \cite{oeis}. Hence there remains the problem of finding $\omega_0(n)_b$ and $c(n)_b$. \cite{tsai22} solves this problem for $b=10$ by associating to each $n$ a periodic function $\mathbb{Z}\to\{0,1\}$ which we describe in the next subsection.

\subsection{Associated periodic function}\label{sec:associated}

Fix a base $b\geq2$ and a natural number $n$ with $b\nmid n$ and $n\neq r_b(n)$ throughout this subsection. To have a clearer picture of the periodic phenomenon illustrated in Theorem \ref{tsai1T}, we define the function $I^n_b\colon \mathbb{N}\to\{0,1\}$ by setting
\begin{equation}
  I^n_b(k)=\begin{cases}
    0\quad\text{if $n(k)_b\notin\mathbb{V}_b$},\\
    1\quad\text{if $n(k)_b\in\mathbb{V}_b$}.
  \end{cases}
\end{equation}
Then by Theorem \ref{tsai1T} $I^n_b$ is a periodic function. It therefore has a unique periodic extension $I^n_b\colon\mathbb{Z}\to\{0,1\}$ which we give the same notation. By \cite[Theorem 11]{tsai22}, $I^n_b$ can be expressed as a linear combination when $b=10$, and the same holds for a general base. We first give notation for certain functions used to form the linear combination.
\begin{definition}
For a natural number $a$, denote by $I_a\colon\mathbb{Z}\to\{0,1\}$ the function defined by
\begin{equation}
  I_a(k)=\begin{cases}
    0\quad\text{if $a\nmid k$}\\
    1\quad\text{if $a\mid k$}.
  \end{cases}
\end{equation}
That is, $I_a$ is the indicator function of $a\mathbb{Z}$ in $\mathbb{Z}$.
\end{definition}
We can now state the linear combination as follows.
\begin{theorem}[{\cite[Theorem 11]{tsai22} for a general base}]
The function $I^n_b$ can be expressed in the form
\begin{equation}\label{eq:expr}
I^n_b=\lambda_1 I_{a_1}+\lambda_2 I_{a_2}+\cdots+\lambda_u I_{a_u},
\end{equation}
where the $u\geq0$, $1\leq a_1<a_2<\cdots<a_u$, and $\lambda_1,\lambda_2,\ldots,\lambda_u\neq0$ are integers.
\end{theorem}
Having expressed the function $I^n_b$ in the form of Equation \eqref{eq:expr}, we use the following result to find $\omega_0(n)_b$ and $c(n)_b$.
\begin{theorem}[{\cite[Corollaries 4 and 5]{tsai22} for a general base}]
The smallest period $\omega_0(n)_b$ and $c(n)_b$ can be found from the expression \eqref{eq:expr} by
\begin{align}
  \omega_0(n)_b&=\lcm\{a_1,a_2,\ldots,a_u\},\\
  c(n)_b&=\inf\{a_1,a_2,\ldots,a_u\}=\begin{cases}
  \infty\quad\text{if $u=0$}\\
    a_1\quad\text{if $u\geq1$}.
  \end{cases}
\end{align}
\end{theorem}
This infimum is thought of as that in the extended real number system.

We did not say how to express $I^n_b$ in the form of Equation \eqref{eq:expr}. A definite procedure for doing this, for $b=10$, is described in \cite{tsai22}, and is easily adapted for a general base.

\subsection{One implies infinitely many.}\label{sec:oneinf}

We show that if there exists a $v$-palindrome in base $b$, then there exist infinitely many.

\begin{theorem}\label{thm:inf}
Let $b\geq2$ be an integer. If there exists a $v$-palindrome in base $b$, then there exist infinitely many $v$-palindromes in base $b$.
\end{theorem}
\begin{proof}
Suppose that $n$ is a $v$-palindrome in base $b$. We have the associated function $I^n_b$ from the previous subsection. If $n$ is a $v$-palindrome in base $b$ that means $I^n_b(1)=1$. Since $I^n_b$ is periodic, say with period $\omega$, we see that
\begin{equation}
I^n_b(1)=I^n_b(1+\omega)=I^n_b(1+2\omega)=\cdots.
\end{equation}
Consequently,
\begin{equation}
n(1)_b,n(1+\omega)_b,n(1+2\omega)_b,\ldots
\end{equation}
are all $v$-palindromes in base $b$.
\end{proof}

\section{\texorpdfstring{$v$}{v}-palindromes in bases \texorpdfstring{$p+1$}{p+1} and \texorpdfstring{$p^2+1$}{p\string^2 + 1}}
\label{sec:vpalb}

In this section we give more examples of $v$-palindromes in bases other than ten. In Subsection \ref{sec:vpalp1}, we give examples of $v$-palindromes in bases $p+1$, for each odd prime $p$. In Subsection \ref{sec:vpalpp1}, we give examples of $v$-palindromes in bases $p^2+1$, for each odd prime $p$. We first prove the following lemmas.

\begin{lemma}
\label{lem:12}
Let $n \in \mathbb{N}$. Then $r_{n+1}(2n) = n^2$.
\end{lemma}
\begin{proof}
We have

\begin{align*}
    r_{n+1}(2n)
        \ &=\  r_{n+1}(2(n+1-1))\\
        &=\  r_{n+1}(2n+2-2) \\
        &=\  r_{n+1}(n+1+(n+1-2)) \\
        &=\  (n+1-2)\cdot (n+1) + 1 \\
        &=\  (n+1)^2 -2(n+1) + 1 \\
        &=\  (n+1-1)^2 \\
&=\  n^2.
\end{align*}
\end{proof}

\begin{lemma}
\label{lem:11}
Let $p$ be an odd prime. Then $v(2p)=v(p^2)$.
\end{lemma}
\begin{proof}
We find that
\begin{equation}
 v(2p)\ =\  v(2) + v(p) \ =\  2 + p \ =\  v(p^2).
\end{equation}
\end{proof}

\subsection{\texorpdfstring{$v$}{v}-palindromes in base \texorpdfstring{$p+1$}{p+1}}\label{sec:vpalp1}
We have the following theorem.
\begin{theorem}
\label{thm1}
Let $p$ be an odd prime. Then $2p\in\mathbb{V}_{p+1}$.
\end{theorem}
\begin{proof}
We have $2p = (1,p-1)_{p+1}$ and so
\begin{equation}\label{eq:reverse}
    r_{p+1}(2p) = (p-1,1)_{p+1} = (p-1)(p+1)+1 = p^2.
\end{equation}
It is clear from the base $p+1$ representation of $2p$, namely $(1,p-1)_{p+1}$, that we have $p+1\nmid 2p$ and $2p \neq r_{p+1}(2p)$. Finally, because $p$ is an odd prime, and using Equation \eqref{eq:reverse}, we have
\begin{equation}
    v(2p) = 2+p = v(p^2) = v(r_{p+1}(2p)).
\end{equation}
This shows that $2p\in\mathbb{V}_{p+1}$.
\end{proof}

Next we consider repeated concatenations of $2p$ in base $p+1$, where $p$ is an odd prime. As in the proof of Theorem \ref{thm1}, both $2p$ and $p^2$ are two digits long in base $p+1$. Using similar notation to \cite{tsai}, we define 
\begin{equation}
\rho_{k,2} \ :=\  \sum_{i=0}^{k-1} (p+1)^{2i} \ =\  1 + \sum_{i=1}^{k-1} (p+1)^{2i},
\end{equation}
for integers $k\geq1$. We find that $2 \mid p + 1$ as $p$ is odd, so $2 \mid \sum_{i=1}^{k-1} (p+1)^{2i}$, and hence $2 \nmid \rho_{k,2}$. We also want to know when $p$ is coprime with $\rho_{k,2}$. 

Note that
\begin{equation}
\label{11}
 \rho_{k,2} \ =\  \sum_{i=0}^{k-1} (p+1)^{2i}\  \equiv\  \sum_{i=0}^{k-1} 1^{2i}\  \equiv\  k \mod{p}.
 \end{equation}
Thus, $p \mid \rho_{k,2}$ if and only if $p \mid k$. We use these two facts to prove the following theorem, recalling that $(2p)(k)_{p+1}$ and $(p^2)(k)_{p+1}$ denote repeated concatenations according to Definition \ref{def:rc}.

\begin{theorem}
Let $p$ be an odd prime and $k\in\mathbb{N}$. Then $(2p)(k)_{p+1} \in \mathbb{V}_{p+1}$ if and only if $p\nmid k$.
\end{theorem}
\begin{proof}
Since $2p = (1,p-1)_{p+1}$, we have $r_{p+1}(2p) = (p-1,1)_{p+1} = p^2$ and that
\begin{equation} \label{eq:11}
   (2p) (k)_{p+1}\  =\  2p\cdot \rho_{k,2},\quad r_{p+1}(2p\cdot \rho_{k,2})\  =\  p^2\cdot \rho_{k,2}.
\end{equation}

Since $\rho_{k,2}$ is an odd number and from Equation \eqref{11}, given $p \nmid k$ we know $p$ is coprime with $\rho_{k,2}$. Using the additive property of $v$, Lemma \ref{lem:11}, and Equation \eqref{eq:11}, we find,
\begin{align*}
    v(2p \cdot \rho_{k,2}) 
        \ &=\  v\left(2p\right) + v(\rho_{k,2}) 
        \\
        &=\  v\left(p^2\right) + v(\rho_{k,2})
        \\
        &=\  v\left(p^2 \cdot \rho_{k,2}\right)\\
        &=\  v\left(r_{p+1}\left(2p \cdot  \rho_{k,2}\right)\right). \\
\end{align*}
Since $2p\cdot \rho_{k,2} = (1,p-1,\dots,1,p-1)_{p+1}$, we know that $p+1 \nmid 2p\cdot \rho_{k,2}$. Lastly, we know that $2p \cdot \rho_{k,2} \neq r_{p+1}\left(2p \cdot \rho_{k,2} \right) = p^2 \cdot \rho_{k,2}$ as $2p \neq p^2$, therefore $(2p)(k)_{p+1} \in \mathbb{V}_{p+1}$.\\

Conversely, assume $(2p)(k)_{p+1} \in \mathbb{V}_{p+1}$. Suppose $p\mid k$. From Equation \eqref{eq:11}, $(2p)(k)_{p+1}=2p \cdot \rho_{k,2}$ and $r_{p+1}((2p)(k)_{p+1})=p^2\cdot  \rho_{k,2}$. From Equation \eqref{11} we know $\rho_{k,2} \equiv  k \mod{p}$. Since $p\mid k$ we have
\begin{equation*}
    \rho_{k,2}\  \equiv\ 0 \mod{p},
\end{equation*}
so $p\mid \rho_{k,2}$. Hence we can rewrite $\rho_{k,2}$ as $\rho_{k,2}=p^a\cdot n$, where $a,n \in \mathbb{N}$ and $p \nmid n$. Note that $\rho_{k,2}$ is odd so $2\nmid n$. Applying $v$ to $(2p)(k)_{p+1}$ we find
\begin{equation*}
    v((2p)(k)_{p+1})\ =\ v(2p\cdot \rho_{k,2})\ =\ v(2\cdot p^{a+1}\cdot n)\ =\ 2+p+a+1+v(n)
\end{equation*}
and
\begin{equation*}
    v(r_{p+1}((2p)(k)_{p+1}))\ =\ v(p^2\cdot \rho_{k,2})\ =\ v(p^{a+2}\cdot n)\ =\ p+a+2+v(n),
\end{equation*}
which contradicts the fact that $(2p)(k)_{p+1} \in \mathbb{V}_{p+1}$. Hence if $(2p)(k)_{p+1} \in \mathbb{V}_{p+1}$, then $p\nmid k$.
\end{proof}

For our last pattern of $v$-palindromes in these bases, we require the following lemma.

\begin{lemma} \label{lem:lngshort}
Let $n,k \in \mathbb{N}$ with $n\geq2$ and $b=n+1$.  Then
\begin{enumerate}
    \item $(1,\underbrace{n,\dots,n}_{k}, n - 1 )_{n+1}\ =\  2n \cdot 1(k+1)_b $, \text{ and}
    \item $(n - 1,\underbrace{n,\dots,n}_{k}, 1 )_{n+1}\  =\  n^2 \cdot 1(k+1)_b$.
\end{enumerate}
\end{lemma}
\begin{proof}
Note that $2n = (1,n-1)_{n+1}$ and $n^2 = (n-1,1)_{n+1}$.  Using this, we find 
\begin{enumerate}
    \item
    \begin{align*}
    2n\cdot 1(k+1)_b 
        \ &=\ (1,n-1)_{n+1}\cdot (\underbrace{1,\cdots,1}_{k+1})_{n+1} \\ 
        &=\ (1,\underbrace{n-1+1, n-1+1,\dots,n-1+1}_k,n-1)_{n+1} \\
        &=\ (1,\underbrace{n,\dots,n}_k,n-1)_{n+1},
    \end{align*}
    \item
    \begin{align*}
    n^2 \cdot 1(k+1)_b
        \ &=\ (n-1,1)_{n+1}\cdot (\underbrace{1,\cdots,1}_{k+1})_{n+1} \\ 
        &=\ (n-1,\underbrace{1+n-1,\dots,1+n-1}_k,1)_{n+1} \\
        &=\ (n-1,\underbrace{n,\dots,n}_k,1)_{n+1}.
    \end{align*}
\end{enumerate}
\end{proof}

We use this to prove the following theorem.
\begin{theorem}
Let $p$ be an odd prime, $k \in \mathbb{N}$, $p \nmid k+1$, and $b=p+1$. Then $2p \cdot 1(k+1)_b\in \mathbb{V}_{p+1}$.
\end{theorem}
\begin{proof}
We find that 
\begin{align*}
    (1, \underbrace{p,\dots,p}_{k}, p - 1)_{p+1}\  &\neq\  r_{p+1}((1, \underbrace{p,\dots,p}_{k}, p - 1)_{p+1})\\
    &=\  (p-1, \underbrace{p,\dots,p}_{k}, 1)_{p+1}
\end{align*}

(this also shows $r_{p+1}(2p \cdot 1(k+1)_{p+1}) = p^2 \cdot 1(k+1)_{p+1}$).  We have
\begin{equation}
    \label{eq:111a}
    1(k+1)_b\  =\  \sum ^{k} _{i=0} (p+1)^i.
\end{equation}
From this, we find
\begin{equation}
    \label{111b}
    1(k+1)_b\  \equiv\  \sum ^{k} _{i=0} (p+1)^i\  \equiv\  \sum ^{k} _{i=0} (1)^i\  \equiv\  k+1 \mod{p}.
\end{equation}
Since $p \nmid k+1$ we know $p \nmid 1(k+1)_b$.  Additionally, we know $2 \nmid  1(k+1)_b$. Further, we find $p+1 \nmid 2 \cdot 1(k+1)_b$ so $p+1 \nmid 2p \cdot 1(k+1)_b$.
Finally we show the numbers are $v$-palindromes by using the additivity of $v$, Lemma \ref{lem:11}, and Equation \eqref{eq:11}. We find that
\begin{align*}
    v((1, \underbrace{p,\dots,p}_{k}, p - 1)_{p+1})
        \ &=\  v\left(2p \cdot 1(k+1)_b\right) \\
        &=\  v(2p) + v(1(k+1)_b) \\
        &=\  v(p^2) + v(1(k+1)_b) \\
        &=\  v(p^2 \cdot 1(k+1)_b) \\
        &=\  v((p-1, \underbrace{p,\dots,p}_{k}, 1)_{p+1}).
\end{align*}
This shows that $2p \cdot 1(k+1)_b\in \mathbb{V}_{p+1}$.
\end{proof}
\subsection{\texorpdfstring{$v$}{v}-palindromes in base \texorpdfstring{$p^2+1$}{p\string^2 + 1}}\label{sec:vpalpp1} 
Recall \cite[Theorem 3]{tsai3}, which applies to a base $b = 3^2 +1$.  We begin by generalizing this Theorem to all bases one greater then an odd prime squared.  We set base $b = p^2 + 1$ as our base, keeping $p$ as an odd prime for the remainder of this section.

\begin{theorem}
\label{thm:genrl3}
Let $p$ be an odd prime.  If $\rho$ is a palindrome in base $b = p^2 +1$ consisting entirely of the digits $0$ and $1$, then $2p^2\rho \in \mathbb{V}_{p^2+1}$.
\end{theorem}
\begin{proof}
We begin by noting that $b\nmid \rho$, since if $\rho$ has a last digit of $0$ then it has a leading digit of $0$, however this means $\rho$ is not a palindrome as any leading digits of $0$ are ignored making $\rho$ of the form $1,\dots,0$. Thus we know the last digit of $\rho$ is $1$.  Further, since we know $p$ is odd, $p+1$ is even thus any number with last digit $1$ is odd, so we know $2\nmid \rho$.

When read from left to right, $\rho$ must be formed by $a_1$ ones, followed by $a_2$ zeros, followed by $a_3$ ones, and so on until lastly, $a_{2r-1}$ ones, where $r,a_1,a_2,\dots,a_{2r-1} \in \mathbb{N}$ such that $ a_i = a_{2r-i}$ for integers $i \in [1,2r-1]$.  Writing out $\rho$ we get

\[\rho = \underbrace{1\dots 1}_{a_1}\overbrace{0\dots 0}^{a_2}\underbrace{1\dots 1}_{a_3}\dots \underbrace{1\dots 1}_{a_3}\overbrace{0\dots 0}^{a_2}\underbrace{1\dots1}_{a_1}.\]
Using the equalities $2p^2 = (1, p^2-1)_{p^2+1}$ and $p^4 = (p^2-1,1)_{p^2+1}$ we find

\begin{multline*}
2p^2\rho = \\ \medmath{(1,\underbrace{p^2,\dots,p^2}_{a_1-1},q,\overbrace{0,\dots,0}^{a_2-1},1,\underbrace{p^2,\dots,p^2}_{a_3-1},q, \dots,1, \underbrace{p^2,\dots,p^2}_{a_3-1} ,q,\overbrace{0,\dots,0}^{a_2-1},1,\underbrace{p^2,\dots,p^2}_{a_1-1},q)_{p^2+1}}
\end{multline*}
and
\begin{multline*}
p^4\rho = \\ \medmath{(q,\underbrace{p^2,\dots,p^2}_{a_1-1},1,\overbrace{0,\dots,0}^{a_2-1},q,\underbrace{p^2,\dots,p^2}_{a_3-1},1, \dots,q, \underbrace{p^2,\dots,p^2}_{a_3-1} ,1,\overbrace{0,\dots,0}^{a_2-1},q,\underbrace{p^2,\dots,p^2}_{a_1-1},1)_{p^2+1}}
\end{multline*}
\noindent where $q=p^2-1$ has been substituted to save space.  From this we clearly see that $p^2+1\nmid 2p^2\rho$ and that $p^4\rho = r_{p^2+1}(2p^2 \rho) \neq 2p^2 \rho$.  Let $\alpha \geq 0$ and $n \geq 1$ be integers such that $\rho = p^\alpha n$ and $(p,n) = 1$. Then

\begin{align*}
    v(2p^2\rho) \ 
        &=\ v(2p^2 \cdot p^\alpha n)\\
        &=\ v(2p^{2+\alpha}n)\\
        &=\ 2+p+2+\alpha+v(n)\\
        &=\ v(p^{4+\alpha}n)\\
        &=\ v(p^4 \cdot p^\alpha n)\\
        &=\ v(r_{p^2+1}(2p^2\rho)).
\end{align*}
This shows that $2p^2 \rho \in \mathbb{V}_{p^2+1}$.

\end{proof}

Next we prove three Corollaries to Theorem \ref{thm:genrl3} that mirror the three theorems proved in Subsection \ref{sec:vpalp1}.

\begin{corollary}
Let $p$ be an odd  prime. Then $2p^2 \in \mathbb{V}_{p^2+1}$.
\end{corollary}
\begin{proof}
Note that $2p^2 = 2p^2 \cdot 1$.  Since $1$ is a palindrome consisting only of the digit $1$, by Theorem \ref{thm:genrl3}, we have $2p^2 \in \mathbb{V}_{p^2+1}$.
\end{proof}

\begin{corollary}
Let $p$ be an odd prime and $k \in \mathbb{N}$. Then $(2p^2) (k)_{p^2+1} \in \mathbb{V}_{p^2+1}$.
\end{corollary}
\begin{proof}
We note that $(2p^2) (k)_{p^2+1} = 2p^2 \cdot \rho_k$, where
\begin{equation}
\rho_k := (\underbrace{1,0,1,0,\dots,0,1,0,1}_{2k-1})_{p^2+1};
\end{equation}
$\rho_k$ is a palindrome consisting entirely of the digits $0$ and $1$. Thus, by Theorem \ref{thm:genrl3} we know $(2p^2)(k)_{p^2+1} \in \mathbb{V}_{p^2+1}$
\end{proof}

\begin{corollary}
Let $p$ be an odd prime, $k \in \mathbb{N}$ and $b=p^2+1$. Then $2p^2 \cdot 1(k+1)_{b} \in \mathbb{V}_{p^2+1}$.
\end{corollary}
\begin{proof}
As $1(k+1)_b$ consists only of the digit $1$, we know it is a palindrome.  Therefore, by Theorem \ref{thm:genrl3} we know $2p^2 \cdot 1(k+1)_{p^2+1} \in \mathbb{V}_{p^2+1}$.

\end{proof}

\section{Existence of \texorpdfstring{$v$}{v}-palindromes for infinitely many bases}\label{sec:main}

In this section we show the existence of $v$-palindromes (and therefore infinitely many $v$-palindromes by Theorem \ref{thm:inf}) for infinitely many bases. Everything is based on the simple fact that $v(5)=v(6)$. Since $v(n)$ is an additive function, for every integer $t\geq1$ with $(t,30)=1$, we have $v(5t)=v(6t)$.

Imagine that we have a base $b\geq2$ for which we would like to show that a $v$-palindrome exists. The first attempt would be to look at two-digit numbers. That is, numbers $(a,c)_b=ab+c$, where $1\leq a<c<b$ are integers. By definition, $(a,c)_b$ is a $v$-palindrome in base $b$ if and only if $v((a,c)_b)=v((c,a)_b)$, or equivalently,
\begin{equation}
  v(ab+c)=v(cb+a).
\end{equation}
This would hold if for some integer $t\geq1$ with $(t,30)=1$,
\begin{equation}\label{eq:system}
\begin{cases}
ab+c=5t,\\
cb+a=6t,
\end{cases}
\end{equation}
simply by the observation in the previous paragraph. To summarize, we have shown the following.
\begin{lemma}
Let $b\geq2$ be an integer. If there exists an ordered triple $(a,c,t)$ of positive integers such that $a<c<b$, $(t,30)=1$, and Equation \eqref{eq:system} holds, then the two-digit number $(a,c)_b$ is a $v$-palindrome in base $b$. Hence, there exists a $v$-palindrome in base $b$.
\end{lemma}
\begin{definition}
We call a triple $(a,c,t)$ in the premise of the above lemma a \emph{permissible triple} for $b$.
\end{definition}
Our strategy is to try to find permissible triples. The system \eqref{eq:system} can be written in matrix from as
\begin{equation}
\begin{pmatrix}
b & 1\\
1 & b
\end{pmatrix}
\begin{pmatrix}
a\\
c
\end{pmatrix}=
t\begin{pmatrix}
5\\
6
\end{pmatrix}.
\end{equation}
Solving this we have
\begin{align}
\begin{pmatrix}
a\\
c
\end{pmatrix}&=t\begin{pmatrix}
b & 1\\
1 & b
\end{pmatrix}^{-1}\begin{pmatrix}
5\\
6
\end{pmatrix}=\frac{t}{b^2-1}
\begin{pmatrix}
b & -1\\
-1 & b
\end{pmatrix}
\begin{pmatrix}
5\\
6
\end{pmatrix}\\
&=\frac{t}{b^2-1}
\begin{pmatrix}
5b-6\\
-5+6b
\end{pmatrix}=
\begin{pmatrix}
\frac{t(5b-6)}{b^2-1}\\
\frac{t(-5+6b)}{b^2-1}
\end{pmatrix}.
\end{align}
We write this separately as
\begin{equation}\label{eq:frac}
  a=\frac{t(5b-6)}{b^2-1},\quad c=\frac{t(6b-5)}{b^2-1},
\end{equation}
from which we also see that $0<a<c$. Hence we have the following lemma.
\begin{lemma}
Let $b\geq2$ be an integer. For every integer $t\geq1$, there exist unique rational numbers $a,c\in\mathbb{Q}$ such that Equation \eqref{eq:system} holds, and they are given by Equation \eqref{eq:frac}. Moreover, $0<a<c$.
\end{lemma}
Hence the only possible permissible triples for $b$ are
\begin{equation}
\left(\frac{t(5b-6)}{b^2-1},\frac{t(-5+6b)}{b^2-1},t\right),
\end{equation}
for an integer $t\geq1$ with $(t,30)=1$. The only missing conditions to fulfill are
\begin{align}
\frac{t(5b-6)}{b^2-1},\frac{t(-5+6b)}{b^2-1}&\in\mathbb{Z},\label{eq:integers}\\
\text{and } \frac{t(-5+6b)}{b^2-1}&<b.
\end{align}
We write
\begin{align}\label{eq:glance}
\frac{t(5b-6)}{b^2-1}&=\frac{t(5b-6)/(5b-6,b^2-1)}{(b^2-1)/(5b-6,b^2-1)},\\
\frac{t(-5+6b)}{b^2-1}&=\frac{t(-5+6b)/(-5+6b,b^2-1)}{(b^2-1)/(-5+6b,b^2-1)}.
\end{align}
Hence we see that Equation \eqref{eq:integers} holds if and only if $t$ is a multiple of
\begin{equation}
  f(b)=\left[\frac{b^2-1}{(5b-6,b^2-1)},\frac{b^2-1}{(-5+6b,b^2-1)}\right];
\end{equation}
here we also defined the function $f(b)$ for integers $b\geq2$. Hence we have shown the following lemma.
\begin{lemma}
Let $b\geq2$ be an integer. Then the permissible triples of $b$ are precisely the triples
\begin{equation}
  \left(\frac{t(5b-6)}{b^2-1},\frac{t(-5+6b)}{b^2-1},t\right),
\end{equation}
where
\begin{equation}
t\in S(b)=\left\{t\in\mathbb{N}\colon (t,30)=1,\, f(b)\mid t,\, t<\frac{b(b^2-1)}{-5+6b}\right\};
\end{equation}
where we also defined the set-valued function $S(b)$ for integers $b\geq2$.
\end{lemma}
However, the above lemma does not promise that permissible triples exist, i.e., $S(b)\neq\varnothing$. However, we can get the following sufficient condition.
\begin{lemma}
Let $b\geq2$ be an integer. If
\begin{equation}
  (f(b),30)=1,\quad f(b)<\frac{b(b^2-1)}{-5+6b},
\end{equation}
then $f(b)\in S(b)$, and consequently there is a permissible triple for $b$.
\end{lemma}
Since $f(b)\mid b^2-1$, if $(b^2-1,30)=1$ then $(f(b),30)=1$. Hence the above lemma can be weakened to the following.
\begin{lemma}
Let $b\geq2$ be an integer. If
\begin{equation}\label{eq:conditions}
  (b^2-1,30)=1,\quad f(b)<\frac{b(b^2-1)}{-5+6b},
\end{equation}
then $f(b)\in S(b)$, and consequently there is a permissible triple for $b$.
\end{lemma}
We now consider the condition $(b^2-1,30)=1$. It is easily shown that this is equivalent to having both $b\equiv \modd{0} {6}$ and $b\equiv \modd{0,2,3} {5}$. In particular, $b\equiv \modd{0} {30}$ is a sufficient condition. Suppose that $k\geq1$ is an integer, then
\begin{align}
f(30k)&=\left[\frac{(30k)^2-1}{(5(30k)-6,(30k)^2-1)},\frac{(30k)^2-1}{(-5+6(30k),(30k)^2-1)}\right]\\
&=\left[\frac{(30k)^2-1}{(6k-2,11)},\frac{(30k)^2-1}{(5k+2,11)}\right],
\end{align}
where for the second equality we used a property of the greatest common divisor function to simplify. Because of the right inequality in Equation \eqref{eq:conditions}, we want $f(30k)$ to be small. Thus it might be good if we have $(6k-2,11)=(5k+2,11)=11$, which is easily shown to be equivalent to $k\equiv \modd{4} {11}$. If we assume that $k\equiv \modd{4} {11}$, then
\begin{equation}
  f(30k)=\frac{(30k)^2-1}{11}.
\end{equation}
On the other hand, the right-hand-side of the right inequality \eqref{eq:conditions} becomes
\begin{equation}
\frac{(30k)((30k)^2-1)}{-5+6(30k)}.
\end{equation}
That $f(30k)$ is strictly less than the above quantity is equivalent to
\begin{equation}
  -5+6(30k)<11(30k),
\end{equation}
which always holds. Hence the above lemma can be further weakened to the following.
\begin{theorem}
Let $k\equiv \modd{4} {11}$ be a positive integer, then
\begin{equation}
  \left(\frac{-6+150k}{11},\frac{-5+180k}{11},\frac{-1+900k^2}{11}\right)
\end{equation}
is a permissible triple for the base $30k$. In particular, the two-digit number
\begin{equation}
\left(\frac{-6+150k}{11},\frac{-5+180k}{11}\right)_{30k}
\end{equation}
is a $v$-palindrome in base $30k$.
\end{theorem}
Hence we have proved the existence of $v$-palindromes for infinitely many bases, summarized as follows.
\begin{corollary}\label{cor:infb}
  If $b\equiv \modd{120} {330}$ is a positive integer, then there exists a $v$-palindrome in base $b$.
\end{corollary}

In particular there is a positive density of bases $b\geq2$ for which a $v$-palindrome exists.

\section{Further problems}\label{sec:further}
In this section we describe some directions for further investigation.

\subsection{Three conjectures}\label{sub:short}
In the short note \cite{tsai18}, three conjectures on $v$-palindromes in base ten have been proposed by commentators and we restate them as follows.

\begin{conjecture}[{\cite[(a)]{tsai18}}]\label{conj:1}
There does not exist a prime $v$-palindrome in base ten.
\end{conjecture}

\begin{conjecture}[{\cite[(b)]{tsai18}}]\label{conj:2}
There are infinitely many $v$-palindromes $n$ in base ten such that both $n$ and $r_{10}(n)$ are square-free.
\end{conjecture}

\begin{conjecture}[{\cite[(c)]{tsai18}}]\label{conj:3}
The only positive integer $n$ such that $n\neq r_{10}(n)$ and $n = v(r(n))$ is $49$.
\end{conjecture}

As mentioned in Subsection \ref{sub:defn}, the prime $v$-palindromes in base ten are characterized in \cite{polymath}. This result, however, does not prove nor disprove Conjecture \ref{conj:1}. Also noted in Subsection \ref{sub:defn} is that there are prime $v$-palindromes in bases $16$ and $276$. Hence we may consider the following problem.

\begin{problem}
Let $b\geq 2$ be an integer.  When does there exist a prime $v$-palindrome in base $b$?
\end{problem}

We may also consider Conjectures \ref{conj:2} and \ref{conj:3} for a general base.

\subsection{Two problems}\label{sub:palin}
While \cite{pong1} provides an exact formula for the number of palindromes up to a given positive integer, the same can be considered for $v$-palindromes, namely the following.
\begin{problem}
Let $b\geq2$ be an integer. Is there a formula for the number of $v$-palindromes in base $b$ up to a given positive integer? If not, how can it be approximated?
\end{problem}
From $199$ until $575$ are $377$ consecutive positive integers which are not $v$-palindromes in base ten. Just as sequences of consecutive composite numbers can be arbitrarily long, we may consider the following problem.
\begin{problem}
    Let $b\geq2$ be an integer. Can a sequence of consecutive positive integers each not a $v$-palindrome in base $b$ be arbitrarily long?
\end{problem}

\subsection{Existence of $v$-palindromes in an arbitrary base}\label{sub:forall}

Section \ref{sec:vpalb} $\\$ showed that $v$-palindromes exist in bases $p+1$ and $p^2+1$ for any odd prime $p$. Section \ref{sec:main} showed that $v$-palindromes exist in all bases $b\equiv \modd{120} {330}$. However we are still left with the problem of determining, for an arbitrary integer $b\geq2$, whether a $v$-palindrome in base $b$ exists.

The proof in the previous section is based on the equality $v(5)=v(6)$. It is conceivable that the same method basing on other common values of $v$ will find other bases $b$ for which a $v$-palindrome exists. For instance, we have
\begin{gather}
v(5)=v(6)=v(8)=v(9),\\
v(7)=v(10)=v(12)=v(18). 
\end{gather}

We give the following table of the smallest $v$-palindrome, i.e., $\min(\mathbb{V}_b)$, for the first few bases, calculated using PARI/GP \cite{pari}.
\begin{table}[H]
 \caption{The smallest $v$-palindrome for bases $b\leq 19$.}
 \label{table:indicatorfun}
 \centering
  \begin{tabular}{lll}
   \hline
   $b$ & $\text{$\min(\mathbb{V}_b)$ written in base $10$}$ & $\text{$\min(\mathbb{V}_b)$ written in base $b$}$ \\
   \hline \hline
   $2$ & $175$ & $1,0,1,0,1,1,1,1$ \\
   $3$ & $1280$ & $1,2,0,2,1,0,2$ \\
   $4$ & $6$ & $1,2$ \\
   $5$ & $288$ & $2,1,2,3$ \\
   $6$ & $10$ & $1,4$ \\
   $7$ & $731$ &  $2,0,6,3$\\
   $8$ & $14$ & $1,6$ \\
   $9$ & $93$ & $1,1,3$ \\
   $10$ & $18$ & $1,8$ \\
   $11$ & $135$ & $1,1,3$ \\

$12$ & $22$ & $1,10$\\
$13$ & $63$ & $4,11$\\

$14$ & $26$ & $1,12$\\
 $15$ & $291$ & $1,4,6$\\
 $16$ & $109$ & $6,13$\\
 $17$ & $581$ & $2,0,3$\\
 $18$ & $34$ & $1,16$\\
 $19$ & $144$ & $7,11$ \\
   \hline
  \end{tabular}
\end{table}

% -----------------
\bibliography{refs}{}
\bibliographystyle{plain}
% -----------------

\end{singlespace}
\end{document}